\documentclass[psamsfont,a4paper,11pt]{amsart}

\usepackage[english]{babel}   

\usepackage{setspace}     
\usepackage{enumerate} 	
\usepackage{afterpage} 	
\usepackage[usenames,dvipsnames]{color}
\usepackage{graphicx}	

\usepackage{amsmath} 	
\usepackage{amssymb}	
\usepackage{mathrsfs} 	
\usepackage{amsfonts} 	
\usepackage{latexsym} 	
\usepackage[latin1]{inputenc} 

\usepackage{amsthm} 	
\usepackage{stackrel}       
\usepackage{amscd} 	

\usepackage{tabularx}	
\usepackage{longtable}	

\usepackage{verbatim} 
\usepackage{blindtext} 

\usepackage{amssymb,latexsym,amsmath}    
\usepackage{pstricks,multido,pst-plot}		 
\usepackage{multicol,fancyheadings}			   
\usepackage{float}                       
\usepackage[vcentermath,enableskew]{youngtab}

\usepackage{pgfplots}
\usepackage{subfig}

\allowdisplaybreaks

\def\d{\delta} 
 
\def\g{\gamma} 
 
\def\l{\lambda}

\def\t{\tau}

\newcommand{\C}{\mathbb{C}}

\newcommand{\NN}{\mathbb{N}} 
\newcommand{\Z}{\mathbb{Z}}

\newtheorem{Thm}{Theorem}[]		
\newtheorem{Lemma}[Thm]{Lemma}

\newtheorem{Cor}[Thm]{Corollary}

\newtheorem{Qn}{Question}
\newtheorem*{thm}{Theorem}	

\newtheorem*{qn}{Question}

\theoremstyle{definition}

\theoremstyle{remark}
 
\newtheorem*{rmk}{Remark}

\newtheorem{ind}[]{{\rm\it Indice}}

\usepackage{multicol}

\usepackage{tikz}
\usetikzlibrary{arrows}

\def\ferrers(#1,#2){%
    \if\space #2
        \put(0,1){\line(1,0){#1}}
        \multiput(0,0)(1,0){#1}{\line(0,1){1}} \put(#1,0){\line(0,1){1}}
        \put(0,0){\line(1,0){#1}}
    \else
        \put(0,1){\line(1,0){#1}}
        \multiput(0,0)(1,0){#1}{\line(0,1){1}} \put(#1,0){\line(0,1){1}}
        \put(0,0){\line(1,0){#1}}
        \put(0,1){\ferrers(#2)}
    \fi
}

\usepackage{paralist}
\usepackage{cite}

\title{Conjugacy growth series for finitary wreath products}

\author[Locus]{Madeline Locus}
\address{Department of Mathematics and Computer Science,
Emory University, Atlanta, GA 30322}
\email{madeline.locus@emory.edu}

\begin{document}

\subjclass[2010]{05A17, 11P81, 20B30, 20B35, 20C32}
\keywords{partitions, group theory, representation theory}

\begin{abstract}
We examine the conjugacy growth series of all wreath products of the finitary permutation groups $\text{Sym}(X)$ and $\text{Alt}(X)$ for an infinite set $X$.  We determine their asymptotics, and we characterize the limiting behavior between the $\text{Alt}(X)$ and $\text{Sym}(X)$ wreath products.  In particular, their ratios form a limit if and only if the dimension of the symmetric wreath product is twice the dimension of the alternating wreath product.
\end{abstract}

\maketitle

\section{Introduction and Statement of Results}

We begin by defining the infinite finitary symmetric and alternating groups and their corresponding wreath products, and then we state our results regarding growth series identities.

For an infinite set $X$, the \emph{finitary symmetric group} $\text{Sym}(X)$ is the group of permutations of $X$ with finite support.  We define the \emph{permutational wreath product} of a group $H$ with $\text{Sym}(X)$ as the group $H\wr_X\text{Sym}(X):=H^{(X)}\rtimes\text{Sym}(X)$ with the following properties:

\begin{enumerate}[(i)]
\item The group $H^{(X)}$ is the group of functions from $X$ to $H$ with finite support.
\item The action of permutations $f\in\text{Sym}(X)$ on functions $\psi\in H^{(X)}$ is defined by
\begin{equation*}
\psi\mapsto f(\psi):=\psi\circ f^{-1}.
\end{equation*}

\item Multiplication in the semi-direct product is defined for $\varphi,\psi\in H^{(X)}$ and $f,g\in\text{Sym}(X)$ by
\begin{equation*}
(\varphi,f)(\psi,g)=(\varphi f(\psi),fg).
\end{equation*}

\end{enumerate}

The \emph{finitary alternating group} $\text{Alt}(X)$ is the subgroup of $\text{Sym}(X)$ of permutations with even signature, and the permutational wreath product $H\wr_X\text{Alt}(X)$ is defined as above.  We now define some general terminology.  For any group $G$ generated by a set $S$, the \emph{word length} $\ell_{G,S}(g)$ of any element $g\in G$ is the smallest nonnegative integer $n$ such that there exist $s_1,\dots,s_n\in S\cup S^{-1}$ with $g=s_1\cdots s_n$.  The \emph{conjugacy length} $\kappa_{G,S}(g)$ is the smallest word length appearing in the conjugacy class of $g$.  If $n$ is any natural number, we denote by $\g_{G,S}(n)\in\NN\cup\{0\}\cup\{\infty\}$ the number of conjugacy classes in $G$ with smallest word length $n$.  If $\g_{G,S}(n)$ is finite for all $n$, then we may define the conjugacy growth series of a group $G$ with generating set $S$ to be the following $q$-series:
\begin{equation*}
C_{G,S}(q):=\sum\limits_{[g]\in\text{Conj}(G)}q^{\kappa_{G,S}(g)}=\sum\limits_{n=0}^\infty\g_{G,S}(n)q^n,
\end{equation*}
where the first sum is over representatives of conjugacy classes of $G$.  Bacher and de la Harpe \cite{BdlH} prove conjugacy growth series identities for sufficiently large\footnote{The condition that the generating sets are sufficiently large refers to the properties defined in Section \ref{2}.} generating sets $S$ of $\text{Sym}(X)$, $S'$ of $\text{Alt}(X)$, and $S^{\left(W_S\right)}$ of $W_S=H_S\wr_X\text{Sym}(X)$ relating the finitary permutation groups and their wreath products to the partition function.  Explicitly, we have the fascinating identities
\begin{equation*}\label{cgs}
C_{\text{Sym}(X),S}(q)=\sum\limits_{n=0}^\infty p(n)q^n=\prod_{n=1}^\infty\frac{1}{1-q^n}\tag{1.1}
\end{equation*}
for the finitary symmetric group,
\begin{equation*}\label{cgsalt}
C_{\text{Alt}(X),S'}(q)=\left(\sum\limits_{n=0}^\infty p(n)q^n\right)\left(\sum\limits_{m=0}^\infty p_e(m)q^m\right)=\frac{1}{2}\prod\limits_{n=1}^\infty\frac{1}{(1-q^n)^2}+\frac{1}{2}\prod\limits_{n=1}^\infty\frac{1}{1-q^{2n}}\tag{1.2}
\end{equation*}
for the finitary alternating group\footnote{Recall that $p_e(m)$ denotes the number of partitions of $m$ into an even number of parts.}, and
\begin{equation*}\label{cgswr}
C_{W_S,S^{\left(W_S\right)}}(q)=\sum\limits_{n=0}^\infty \g_{W_S,S^{\left(W_S\right)}}(n)q^n=\prod\limits_{n=1}^\infty\frac{1}{\left(1-q^n\right)^{M_S}}\tag{1.3}
\end{equation*}
for wreath products $W_S=H_S\wr_X\text{Sym}(X)$, where $M_S$ is the number of conjugacy classes of $H_S$.  Following Bacher's and de la Harpe's proofs of these identities, we prove\footnote{See also Ian Wagner's work on properties of $\text{Alt}(X).$} the corresponding growth series identity for a sufficiently large generating set $S^{\left(W_A\right)}$ of $W_A=H_A\wr_X\text{Alt}(X)$, namely
\begin{equation*}\label{cgswralt}
C_{W_A,S^{\left(W_A\right)}}(q)=\sum\limits_{n=0}^\infty\g_{W_A,S^{\left(W_A\right)}}(n)q^n=\left(\frac{1}{2}\prod_{n=1}^\infty\frac{1}{(1-q^n)^2}+\frac{1}{2}\prod_{n=1}^\infty\frac{1}{1-q^{2n}}\right)^{M_A}\tag{1.4}
\end{equation*}
where $M_A$ is the number of conjugacy classes of $H_A$.  We provide the proof of equation (\ref{cgswralt}) in Section \ref{2}.  From now on, we denote $\g_W(n):=\g_{W,S^{(W)}}(n)$ for convenience.\\

\begin{rmk}
Recall Dedekind's eta function $\eta(\t)=q^{1/24}\prod_{n\geq1}(1-q^n)$ for $\t\in\mathcal{H}$, where $\mathcal{H}$ denotes the upper half complex plane and $q:=e^{2\pi i\t}$.  Equation (\ref{cgsalt}) can be written as the linear combination of eta-quotients
\begin{equation*}
C_{\text{Alt}(X),S'}(q)=\frac{1}{2}\cdot\frac{q^{1/12}}{\eta(\t)^2}+\frac{1}{2}\cdot\frac{q^{1/12}}{\eta(2\t)},
\end{equation*}
which is essentially the sum of a modular form of weight $-1$ and a modular form of weight $-\frac{1}{2}$, up to multiplication by $q^{1/12}$.  Studying such linear combinations may shed light on properties of sums of mixed weight modular forms.
\end{rmk}

It is natural to consider the number $\g_{W_S}(n)$ as a function of the number of conjugacy classes $M_S$ in order to study properties of the coefficients of the above $q$-series.  Here we obtain a universal recurrence for these numbers.  This result requires the ordinary divisor function $\sigma_k(n)=\sum_{d\mid n}d^k$.

\begin{Thm}\label{thm}
For $n\geq2$, define the polynomial
\begin{equation*}
\widehat{F}_n(x_1,\dots,x_{n-1}):=\sum_{\substack{m_1,\dots,m_{n-1}\geq0 \\
m_1+\cdots+(n-1)m_{n-1}=n} }(-1)^{m_1+\cdots+m_{n-1}}\cdot\frac{(m_1+\cdots+m_{n-1}-1)!}{m_1!\cdots m_{n-1}!}\cdot x_1^{m_1}\cdots x_{n-1}^{m_{n-1}}.
\end{equation*}
Let $H_S$ be a finite group with $M_S$ conjugacy classes, $X$ an infinite set, and $W_S=H_S\wr_X\emph{Sym}(X)$ a wreath product generated by a sufficiently large set $S^{\left(W_S\right)}$.  Then we have
\begin{equation*}
C_{W_S,S^{\left(W_S\right)}}(q)=\sum\limits_{n=0}^\infty\g_{W_S}(n)q^n=\prod\limits_{n=1}^\infty\left(1-q^n\right)^{-M_S},
\end{equation*}
where $\g_{W_S}(n)$ satisfies the recurrence relation
\begin{equation*}
\g_{W_S}(n)=\widehat{F}_n\big(\g_{W_S}(1),\dots,\g_{W_S}(n-1)\big)+\frac{M_S}{n}\cdot\sigma_1(n).
\end{equation*}

\end{Thm}

\begin{rmk}
The polynomials $\widehat{F}_n$ are fairly straightforward to compute using only the partitions of $n$; the first few are listed below.
\begin{align*}
\widehat{F}_2(x_1)&=\frac{1}{2}x_1^2\\
\widehat{F}_3(x_1,x_2)&=-\frac{1}{3}x_1^3+x_1x_2\\
\widehat{F}_4(x_1,x_2,x_3)&=\frac{1}{4}x_1^4-x_1^2x_2+\frac{1}{2}x_2^2+x_1x_3\\
\widehat{F}_5(x_1,x_2,x_3,x_4)&=-\frac{1}{5}x_1^5+x_1^3x_2-x_1^2x_3-x_1x_2^2+x_1x_4+x_2x_3\\
\widehat{F}_6(x_1,x_2,x_3,x_4,x_5)&=\frac{1}{6}x_1^6-x_1^4x_2+x_1^3x_3+\frac{3}{2}x_1^2x_2^2-x_1^2x_4-2x_1x_2x_3+x_1x_5-\frac{1}{3}x_2^3+x_2x_4+\frac{1}{2}x_3^2.
\end{align*}
\end{rmk}
\begin{rmk}
These polynomials have been used in earlier work \cite{BKO,FL} on divisors of modular forms and the Rogers-Ramanujan identities.
\end{rmk}

In recent work, Nekrasov and Okounkov obtained a different formula for the infinite products in Theorem \ref{thm} in terms of hook lengths of partitions.  Let $\l\vdash L$ denote that $\l$ is a partition of the number $L$.  The hook length of a partition $\l=(\l_1,\dots,\l_n)\vdash L$ is defined using the Ferrers diagram of $\l$.  For example, Figure 1 below is a Ferrers diagram of the partition $\l=(6,4,3,1,1)\vdash15$, Figure 2 represents a hook length of 4, and Figure 3 shows all hook lengths associated to $\l$.
\begin{figure}[h]
	\begin{minipage}{.34\textwidth}
	\includegraphics[width=.55cm,height=.5cm]{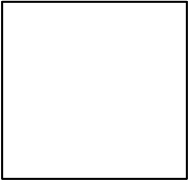}\,\includegraphics[width=.55cm,height=.5cm]{square1}\,\includegraphics[width=.55cm,height=.5cm]{square1}\,\includegraphics[width=.55cm,height=.5cm]{square1}\,\includegraphics[width=.55cm,height=.5cm]{square1}\,\includegraphics[width=.55cm,height=.5cm]{square1}\\
	\includegraphics[width=.55cm,height=.5cm]{square1}\,\includegraphics[width=.55cm,height=.5cm]{square1}\,\includegraphics[width=.55cm,height=.5cm]{square1}\,\includegraphics[width=.55cm,height=.5cm]{square1}\\
	\includegraphics[width=.55cm,height=.5cm]{square1}\,\includegraphics[width=.55cm,height=.5cm]{square1}\,\includegraphics[width=.55cm,height=.5cm]{square1}\\
	\includegraphics[width=.55cm,height=.5cm]{square1}\\
	\includegraphics[width=.55cm,height=.5cm]{square1}
	\caption{Partition}
	\label{partition}
	\end{minipage}%
	\begin{minipage}{.34\textwidth}
	\includegraphics[width=.55cm,height=.5cm]{square1}\,\includegraphics[width=.55cm,height=.5cm]{square1}\,\includegraphics[width=.55cm,height=.5cm]{square1}\,\includegraphics[width=.55cm,height=.5cm]{square1}\,\includegraphics[width=.55cm,height=.5cm]{square1}\,\includegraphics[width=.55cm,height=.5cm]{square1}\\
	\includegraphics[width=.55cm,height=.5cm]{square1}\,\includegraphics[width=.55cm,height=.5cm]{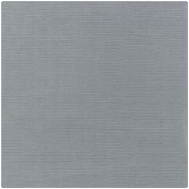}\,\includegraphics[width=.55cm,height=.5cm]{grey_square}\,\includegraphics[width=.55cm,height=.5cm]{grey_square}\\
	\includegraphics[width=.55cm,height=.5cm]{square1}\,\includegraphics[width=.55cm,height=.5cm]{grey_square}\,\includegraphics[width=.55cm,height=.5cm]{square1}\\
	\includegraphics[width=.55cm,height=.5cm]{square1}\\
	\includegraphics[width=.55cm,height=.5cm]{square1}
	\caption{Hook Length}
	\label{partition}
	\end{minipage}%
	\begin{minipage}{.34\textwidth}
	\fbox{10}\,\fbox{7 }\,\fbox{6 }\,\fbox{4 }\,\fbox{2 }\,\fbox{1 }\\
	\fbox{7 }\,\,\fbox{4 }\,\fbox{3 }\,\fbox{1 }\\
	\fbox{5 }\,\,\fbox{2 }\,\fbox{1 }\\
	\fbox{2 }\\
	\fbox{1 }
	\caption{Hook Lengths}
	\label{partition}
	\end{minipage}
\end{figure}

More generally, for each box $v$ in the Ferrers diagram of a partition $\l$, its \emph{hook length} $h_v(\l)$ is defined as the number of boxes $u$ such that
\begin{enumerate}[(i)]
\item $u=v$,
\item $u$ is in the same column as $v$ and below $v$, or
\item $u$ is in the same row as $v$ and to the right of $v$.
\end{enumerate}
The \emph{hook length multi-set} $\mathcal{H}(\l)$ is the set of all hook lengths of $\l$.
Theorem \ref{thm} implies the following formula for $\g_{W_S}(n)$ in terms of hook lengths. 

\begin{Cor}\label{NO}
We have that
\begin{eqnarray*}
\g_{W_S}(n)&=&\widehat{F}_n\big(\g_{W_S}(1),\dots,\g_{W_S}(n-1)\big)+\frac{M_S}{n}\cdot\sigma_1(n)\\
&=&\sum\limits_{\l\vdash n}\prod\limits_{h\in\mathcal{H}(\l)}\left(1+\frac{M_S-1}{h^2}\right).
\end{eqnarray*}
\end{Cor}

\begin{rmk}
Kostant observed \cite{K} that the coefficients of the Nekrasov-Okounkov hook length identity are polynomials in the variable $z=1-M_S$, but he did not give an explicit formula for computing them.
\end{rmk}

In analogy with the previous theorem, one may ask if the coefficients $\g_{W_A}(n)$ in the alternating case can be seen as a function of the number of conjugacy classes $M_A$.  We obtain a similar recurrence relation in this case.

\begin{Thm}\label{thm2}
Let $\widehat{F}_n(x_1,\dots,x_{n-1})$ be defined as above.  Let $H_A$ be a finite group with $M_A$ conjugacy classes, $X$ an infinite set, and $W_A=H_A\wr_X\emph{Alt}(X)$ a wreath product generated by a sufficiently large set $S^{\left(W_A\right)}$.  Then we have
\begin{equation*}
C_{W_A,S^{\left(W_A\right)}}(q)=\sum\limits_{n=0}^\infty\g_{W_A}(n)q^n=\left(\frac{1}{2}\prod_{n=1}^\infty\frac{1}{\left(1-q^n\right)^2}+\frac{1}{2}\prod_{n=1}^\infty\frac{1}{1-q^{2n}}\right)^{M_A},
\end{equation*}
where $\g_{W_A}(n)$ satisfies the recurrence relation
\begin{equation*}
\g_{W_A}(n)=\frac{1}{2^{M_A}}\sum_{k=0}^{M_A}\binom{M_A}{k}\Bigg(\widehat{F}_n\big(a_k(1),\dots,a_k(n-1)\big)-\sum_{\d\mid n}\d\cdot\left[(-1)^\d(k-M_A)-(k+M_A)\right]\Bigg),
\end{equation*}
and the $a_k$ are defined by their generating function
\begin{equation*}
\sum_{n=0}^\infty a_k(n)q^n:=\prod_{n=1}^\infty\frac{1}{\left(1-q^n\right)^{2k}\left(1-q^{2n}\right)^{M_A-k}}.
\end{equation*}
\end{Thm}
\begin{rmk}
It may be possible to interpret the coefficients $\g_{W_A}(n)$ in terms of hook lengths from formulas of Han \cite{H} or others, as in the symmetric case.  The author does not make this connection here.
\end{rmk}

\vspace{.3cm}
It is also natural to study the \emph{exponential rate of conjugacy growth}\footnote{Cotron, Dicks, and Fleming \cite{REU} modify Bacher's and de la Harpe's definition \cite{BdlH} by changing the denominator from $n$ to $\sqrt{n}$.  With denominator $n$, most of the growth series that we study have exponential rate of conjugacy growth zero.} of a group $G$ generated by a set $S$, namely
\begin{equation*}
\widetilde{H}_{G,S}^{\text{conj}}=\limsup_{n\rightarrow\infty}\frac{\log\g_{G,S}(n)}{\sqrt{n}}.
\end{equation*}
It is useful to notice that $\text{exp}\left(\widetilde{H}_{G,S}^{\text{conj}}\right)$ is the radius of convergence of the conjugacy growth series $C_{G,S}(q)$.  For permutational wreath products, we apply a theorem of Cotron, Dicks and Fleming \cite{REU} on the asymptotic behavior of the generalized partition function (see equations (\ref{gpf}) and (\ref{asymp})).  Let $W_S=H_S\wr_X\text{Sym}(X)$ be a wreath product where $H_S$ is a finite group, $M_S$ is the number of conjugacy classes of $H_S$, and $X$ is an infinite set.  It is easy to see from equation (\ref{cgswr}) that the conjugacy growth series of such a wreath product is the generating function of the generalized partition function $p(n)_\textbf{e}$ for the vector $\textbf{e}=(M_S)$.  This implies the following corollary.
\begin{Cor}\label{asymptotic}
Let $W_S=H_S\wr_X\text{Sym}(X)$ be a wreath product where $H_S$ is a finite group, $M_S$ is the number of conjugacy classes of $H_S$, and $X$ is an infinite set.  If $S^{\left(W_S\right)}$ is a sufficiently large generating set of $W_S$, then we have
\begin{equation*}\label{symasymptotic}
\g_{W_S}(n)\sim\left(\frac{M_S^{\frac{1+M_S}{4}}}{2^{\frac{5+3M_S}{4}}3^{\frac{1+M_S}{4}}n^{\frac{3+M_S}{4}}}\right)e^{\pi\sqrt{\frac{2nM_S}{3}}}.
\end{equation*}
\end{Cor}
We now give the exponential rate of conjugacy growth for wreath products in the symmetric case using this asymptotic formula.
\begin{Cor}
The exponential rate of conjugacy growth for the group $W_S=H_S\wr_X\text{Sym}(X)$ defined above is
\begin{eqnarray*}
\widetilde{H}^{\mathrm{conj}}_{W_S}=\pi\sqrt{\frac{2M_S}{3}}.
\end{eqnarray*}
\end{Cor}
We can also apply the theorem to wreath products in the alternating case using equation (\ref{cgswralt}).
\begin{Cor}\label{altasymptotic}
Let $W_A=H_A\wr_X\text{Alt}(X)$ be a wreath product where $H_A$ is a finite group, $M_A$ is the number of conjugacy classes of $H_A$, and $X$ is an infinite set.  If $S^{\left(W_A\right)}$ is a sufficiently large generating set of $W_A$, then we have
\begin{equation*}
\g_{W_A}(n)\sim\left(\frac{M_A^{\frac{1+2M_A}{4}}}{2^{1+2M_A}3^{\frac{1+2M_A}{4}}n^{\frac{3+2M_A}{4}}}\right)e^{2\pi\sqrt{\frac{nM_A}{3}}}.
\end{equation*}
\end{Cor}
We also give the exponential rate of conjugacy growth in the alternating case using the above asymptotic formula.
\begin{Cor}
The exponential rate of conjugacy growth for the group $W_A=H_A\wr_X\text{Alt}(X)$ defined above is
\begin{eqnarray*}
\widetilde{H}^{\mathrm{conj}}_{W_A}=2\pi\sqrt{\frac{M_A}{3}}.
\end{eqnarray*}
\end{Cor}
\vspace{.2cm}
We are interested in finding relationships between wreath products of $\text{Sym}(X)$ and wreath products of $\text{Alt}(X)$.  Let $W_S=H_S\wr_X\text{Sym}(X)$ and $W_S'=H_S'\wr_X\text{Sym}(X)$ be two wreath products of $\text{Sym}(X)$, where $H_S,H_S'$ are finite groups and $M_S,M_S'$ are the number of conjugacy classes of $H_S,H_S'$ respectively.  Let $W_A=H_A\wr_X\text{Alt}(X)$ and $W_A'=H_A'\wr_X\text{Alt}(X)$ be two wreath products of $\text{Alt}(X)$, where $H_A, H_A'$ are finite groups and $M_A, M_A'$ are the number of conjugacy classes of $H_A,H_A'$ respectively.
\begin{Qn}\label{ratioqn}
What is the asymptotic behavior of the following ratios?
\begin{equation*}
(1)\hspace{.3cm}\frac{\g_{W_S}(n)}{\g_{W_S'}(n)}\hspace{1.7cm}(2)\hspace{.3cm}\frac{\g_{W_S}(n)}{\g_{W_A}(n)}\hspace{1.7cm}(3)\hspace{.3cm}\frac{\g_{W_A}(n)}{\g_{W_S}(n)}\hspace{1.7cm}(4)\hspace{.3cm}\frac{\g_{W_A}(n)}{\g_{W_A'}(n)}
\end{equation*}
In particular, when do the ratios approach some nonzero finite number?
\end{Qn}
The asymptotic behavior of the ratios follows from Corollaries \ref{symasymptotic} and \ref{altasymptotic}.
\begin{Cor}\label{ratio}
Let $W_S,W_S',W_A,$ and $W_A'$ be groups as above.  Then as $n\rightarrow\infty$, we have
\begin{eqnarray*}
(1)\hspace{.3cm}\frac{\g_{W_S}(n)}{\g_{W_S'}(n)}&\sim&\left(\frac{M_S^{\frac{1+M_S}{4}}}{M_S'^{\frac{1+M_S'}{4}}}\right)\left[2^{\frac{3}{4}\left(M_S'-M_S\right)}(3n)^{\frac{M_S'-M_S}{4}}\right]e^{\pi\sqrt{\frac{2n}{3}}\left(\sqrt{M_S}-\sqrt{M_S'}\right)}.\\\\
(2)\hspace{.3cm}\frac{\g_{W_S}(n)}{\g_{W_A}(n)}&\sim&\left(\frac{M_S^{\frac{1+M_S}{4}}}{M_A^{\frac{1+2M_A}{4}}}\right)\left[2^{\frac{8M_A-3M_S-1}{4}}(3n)^{\frac{2M_A-M_S}{4}}\right]e^{\pi\sqrt{\frac{2n}{3}}\left(\sqrt{M_S}-\sqrt{2M_A}\right)}.\\\\
(3)\hspace{.3cm}\frac{\g_{W_A}(n)}{\g_{W_S}(n)}&\sim&\left(\frac{M_A^{\frac{1+2M_A}{4}}}{M_S^{\frac{1+M_S}{4}}}\right)\left[2^{\frac{1+3M_S-8M_A}{4}}(3n)^{\frac{M_S-2M_A}{4}=}\right]e^{\pi\sqrt{\frac{2n}{3}}\left(\sqrt{2M_A}-\sqrt{M_S}\right)}.\\\\
(4)\hspace{.3cm}\frac{\g_{W_A}(n)}{\g_{W_A'}(n)}&\sim&\left(\frac{M_A^{\frac{1+2M_A}{4}}}{M_A'^{\frac{1+2M_A'}{4}}}\right)\left[4^{\left(M_A'-M_A\right)}(3n)^{\frac{M_A'-M_A}{2}}\right]e^{2\pi\sqrt{\frac{n}{3}}\left(\sqrt{M_A}-\sqrt{M_A'}\right)}.
\end{eqnarray*}
\end{Cor}
We now observe for which pairs $(M_S,M_S'),(M_S,M_A),(M_A,M_S),$ and $(M_A,M_A')$ these ratios asymptotically approach zero, infinity, or some nonzero finite number.  Corollary \ref{L} follows from the asymptotic behavior of the exponential functions in the above proposition.
\begin{Cor}\label{L}
Let $W_S,W_S',W_A,$ and $W_A'$ be groups as above.  Then as $n\rightarrow\infty$, we have the following asymptotic behavior.
\begin{enumerate}
\item If $M_S<M_S'$, then $\frac{\g_{W_S}(n)}{\g_{W_S'}(n)}\sim0$.  If $M_S>M_S'$, then $\frac{\g_{W_S}(n)}{\g_{W_S'}(n)}\sim\infty$.\\
If $M_S=M_S'$, then $\frac{\g_{W_S}(n)}{\g_{W_S'}(n)}\sim1$.\\\\
\item If $M_S<2M_A$, then $\frac{\g_{W_S}(n)}{\g_{W_A}(n)}\sim0$.  If $M_S>2M_A$, then $\frac{\g_{W_S}(n)}{\g_{W_A}(n)}\sim\infty$.\\
If $M_S=2M_A$, then $\frac{\g_{W_S}(n)}{\g_{W_A}(n)}\sim2^{M_A}$.\\\\
\item If $2M_A<M_S$, then $\frac{\g_{W_A}(n)}{\g_{W_S}(n)}\sim0$.  If $2M_A>M_S$, then $\frac{\g_{W_A}(n)}{\g_{W_S}(n)}\sim\infty$.\\
If $2M_A=M_S$, then $\frac{\g_{W_A}(n)}{\g_{W_S}(n)}\sim\frac{1}{2^{M_A}}$.\\\\
\item If $M_A<M_A'$, then $\frac{\g_{W_A}(n)}{\g_{W_A'}(n)}\sim0$.  If $M_A>M_A'$, then $\frac{\g_{W_A}(n)}{\g_{W_A'}(n)}\sim\infty$.\\
If $M_A=M_A'$, then $\frac{\g_{W_A}(n)}{\g_{W_A'}(n)}\sim1$.
\end{enumerate}
Moreover, the converses of all of the above statements hold as well.
\end{Cor}
Given any two wreath products of $\text{Sym}(X)$ or $\text{Alt}(X)$, the above theorem guarantees the asymptotic behavior of the ratios between the coefficients of their conjugacy growth series.  In other words, for any two wreath products $W$ and $W'$, we know the expected relationship between the number of conjugacy classes of $H$ in $W$ and the number of conjugacy classes of $H'$ in $W'$ with minimal word length $n$ for any $n$.
\begin{rmk}
Although we know the asymptotic behavior of the above ratios, this does not mean that the ratios of the coefficients are always exactly equal to the above values.
\end{rmk}
For example, consider the wreath products $W_S=H_S\wr_X\text{Sym}(X)$ and $W_A=H_A\wr_X\text{Alt(X)}$, where $H_S,H_A$ are finite groups with $M_S=10,M_A=5$ conjugacy classes respectively.  We expect the ratio of the coefficients of $W_S$ to the coefficients of $W_A$ to be asymptotic to $2^5=32$.  We compute the following coefficients with Maple.\\
\begin{center}
\begin{tabular}{ |c|c|c|c| } 
 \hline
$n$ & $\g_{W_S}(n)$ & $\g_{W_A}(n)$ & $\frac{\g_{W_S}(n)}{\g_{W_A}(n)}$ \\ 
 \hline
 \hline
 1 & 10 & 5 & 2\\
 \hline
 10 & 1605340 & 176963 & 9.071613840\\
 \hline
100 & $0.2333013623\times10^{28}$ & $0.7541087996\times10^{26}$ & 30.93736108\\ 
 \hline
200 & $0.1067904403\times10^{42}$ & $0.3346942881\times10^{40}$ & 31.90686071\\ 
 \hline
300 & $0.4721905614\times10^{52}$ & $0.1476229954\times10^{51}$ & 31.98624714\\
 \hline
 400 & $0.5248644122\times10^{61}$ & $0.1640339890\times10^{60}$ & 31.99729613\\
 \hline
500 & $0.5369981415\times10^{69}$ & $0.1678152777\times10^{68}$ & 31.99935959\\
 \hline
\end{tabular}
\end{center}

\vspace{1cm}

\section*{Acknowledgements}
The author would like to thank Ken Ono for his invaluable advice and guidance on this project.

\vspace{1cm}

\section{Proofs}\label{2}

We give the proofs of equation (\ref{cgswralt}) and Theorems \ref{thm} and \ref{thm2} here.  We also explain what it means for a generating set to be sufficiently large and give remarks on Corollaries \ref{NO} and \ref{altasymptotic}.\\

A set $S$ of transpositions of a set $X$ is called \emph{partition-complete} (\emph{PC}) \cite{BdlH} if\\
\begin{enumerate}[(i)]
\item the transposition graph $\Gamma(S)$ is connected, and
\item for every partition $\l=(\l_1,\dots,\l_k)\vdash L$, $\Gamma(S)$ contains a forest of $k$ trees with $\l_1+1,\dots,\l_k+1$ vertices respectively.\\
\end{enumerate}
For the corresponding property of \emph{partition-complete for wreath products} (\emph{PCwr}) \cite{BdlH}, we must first establish more notation.  Let $X$ be an infinite set, $H$ a group, and $W=H\wr_X\text{Sym}(X)$.  The group $W$ acts naturally on the set $H\times X$; namely, for $(\varphi,f)\in W$, the action is defined by
\begin{equation*}
(h,x)\mapsto\big(\varphi(f(x))h,f(x)\big).
\end{equation*}
For $a\in H\setminus\{1\}$ and $u\in X$, we let $\varphi_u^a\in W$ denote the permutation that maps $(h,x)\in H\times X$ to $(ah,u)$ if $x=u$, and to $(h,x)$ otherwise.  Then $(\varphi_u^a)_{a\in H\setminus\{1\},\,u\in X}$ generates the group $H^{(X)}$.  Now, let $H_u:=\{\varphi_u^a\mid a\in H\setminus\{1\}\}$, and define the subsets
\begin{eqnarray*}
T_H&:=&\bigcup_{u\in X}H_u\subseteq H^{(X)},\\
T_X&:=&\{(x\,\,\,y)\in\text{Sym}(X):x,y\in X\text{ are distinct}\}\subseteq\text{Sym}(X).
\end{eqnarray*}
Let $S_H\subset T_H$ and $S_X\subset T_X$ be subsets, and let $S=S_H\sqcup S_X\subseteq W$.  Such a set $S$ is said to be \emph{PCwr} if\\
\begin{enumerate}[(i)]
\item the transposition graph $\Gamma(S_X)$ is connected, and
\item for all $L\geq0$ and partitions $\l=(\l_1,\dots,\l_k)\vdash L$, $\Gamma(S_X)$ contains a forest of $k$ trees $T_1,\dots,T_k$, with $T_i$ having $\l_i$ vertices, including one vertex $x^{(i)}$ such that $\varphi_{x^{(i)}}^a\in S_H$ for all $a\in H\setminus\{1\}$.\\
\end{enumerate}
\begin{rmk}
The conditions \emph{PC} and \emph{PCwr} essentially require the generating set $S$ to contain ``enough" transpositions to represent all possible partitions in its transposition graph.
\end{rmk}

\vspace{.3cm}

\begin{proof}[Proof of equation (\ref{cgswralt})]
This proof follows from the proofs of equations (\ref{cgsalt}) and (\ref{cgswr}) in \cite{BdlH}.  For each $w=(\phi, \sigma)\in W_A = H_A \wr_{X} \text{Alt}(X)$, we can split $\sigma$ into a product of an even number of cycles of even length, denoted $\sigma_{e}$, and a product of cycles of odd length, denoted $\sigma_{o}$, so that $w = (\phi, \sigma_{e} \sigma_{o})$.  Let ${(H_A)}_{*}$ denote the set of conjugacy classes of $H_A$; we write $1\in{(H_A)}_{*}$ for the class $\{1\}\in H_A$.  To each conjugacy class in $W_A$ we associate an ${(H_A)}_{*}$-indexed family of partitions.  Using the same notation as in \cite{BdlH}, we associate the conjugacy classes in $H_A$ to the family of partitions $$\left(\lambda^{(1)}, \nu^{(1)}; \left(\mu^{(\eta)}, \gamma^{(\eta)}\right)_{\eta \in {(H_A)}_{*} \setminus 1} \right),$$
where $\nu^{(1)}$ and $\gamma^{(\eta)}$ each have an even number of positive parts, in the following way.

Let $X^{(w)}$ be the finite subset of $X$ that is the union of the supports of $\phi$ and $\sigma$.  Let $\sigma$ be the product of the disjoint cycles $c_{1}, ..., c_{k}$,  where $c_{i} = \left(x_{1}^{(i)}, x_{2}^{(i)}, ..., x_{v_{i}}^{(i)}\right)$ with $x_{j}^{(i)} \in X^{(w)}$ and $v_{i} = \text{ length}(c_{i})$.  We include cycles of length $1$ for each $x \in X$ such that $x \in\text{sup}(\phi)$ and $x \notin\text{sup}(\sigma)$, so that $$X^{(w)} = \bigsqcup_{1\leq i \leq k}\text{sup}(c_{i}).$$  Define $\eta_{*}^{w}(c_{i}) \in {(H_A)}_{*}$ to be the conjugacy class of the product $\phi\left(x_{v_{i}}^{(i)}\right) \phi\left(x_{v{i}-1}^{(i)}\right) \cdots \phi\left(x_{1}^{(i)}\right) \in H_A$.  For $\eta \in {(H_A)}_{*}$ and $\ell \geq 1$, let $m_{\ell}^{w,\eta}$ denote the number of cycles $c$ in $\{c_{1}, ..., c_{k} \}$ such that $\text{length}(c)=\ell$ and $\eta_{*}^{w}(c) = \eta$.  Let $\mu^{w, \eta} \vdash n^{w, \eta}$ be the partition with $m_{\ell}^{w, \eta}$ parts equal to $\ell$, for all $\ell \geq 1$.  Note that $$\sum_{\eta \in {(H_A)}_{*}} n^{w, \eta} = \sum_{\eta \in {(H_A)}_{*},\,\,\ell \geq 1} \ell m_{\ell}^{w, \eta} = \Big|X^{(w)}\Big|.$$  Also observe that the partition $\mu^{w, 1}$ does not contain parts of size $1$, because if $v_{i} = 1$, then $\eta_{*}^{w}(c_{i}) \neq 1$.  Using the same notation as above, let $\lambda^{w, 1}$ be the partition with $m_{\ell}^{w, 1}$ parts equal to $\ell - 1$.  We can write $\sigma = \sigma_{e} \sigma_{o}$ as above, so $\l^{w,1}$ splits into two partitions, one of which has an even number of parts.  Define the \emph{type} of $w$ to be the family $\left(\lambda^{(1)}, \nu^{(1)}; \left(\mu^{(\eta)}, \gamma^{(\eta)}\right)_{\eta \in {(H_A)}_{*} \setminus 1} \right)$.  Then two elements in $W_A$ are conjugate if and only if they have the same type.  Thus, each ${(H_A)}_{*}$-indexed family of partitions $\left(\lambda^{(1)}, \nu^{(1)}; \left(\mu^{(\eta)}, \gamma^{(\eta)}\right)_{\eta \in {(H_A)}_{*} \setminus 1} \right)$ is the type of one conjugacy class in $W_A$.

Consider an ${(H_A)}_{*}$-indexed family of partitions $\left(\lambda^{(1)}, \nu^{(1)}; \left(\mu^{(\eta)}, \gamma^{(\eta)}\right)_{\eta \in {(H_A)}_{*} \setminus 1} \right)$ and the corresponding conjugacy class in $W_A$.  Let $u^{(1)}, v^{(1)}, u^{(\eta)}, v^{(\eta)}$ be the sums of the parts of $\lambda^{(1)}, \nu^{(1)}, \mu^{(\eta)}, \gamma^{(\eta)}$ respectively, and let $k^{(1)}, t^{(1)}, k^{(\eta)}, t^{(\eta)}$ be the number of parts of  $\lambda^{(1)}, \nu^{(1)}, \mu^{(\eta)}, \gamma^{(\eta)}$ respectively. 

Choose a representative $w = (\phi, \sigma)$ of this conjugacy class such that
$$\sigma = \prod_{i=1}^k c_{i} = \prod_{i=1}^k \left(x_{1}^{(i)}, x_{2}^{(i)}, ..., x_{\mu_{i}}^{(i)}\right)$$ and 
\begin{eqnarray*}
\phi\left(x_{j}^{(i)}\right) &=& 1 \in H_A \text{ for all } j \in \{1, ..., \mu_{i} \}  \quad \text{ when }  \eta_{*}^{w}(c_{i}) = 1,\\
\phi\left(x_{j}^{(i)}\right) &=& \begin{cases}
1 \text{ for all } j \in \{1, ..., \mu_{i} - 1\} \\
h \neq 1 \text{ for } j = \mu_{i} 
\end{cases}  \text{ when } \eta_{*}^{w}(c_{i}) \neq 1.
\end{eqnarray*}

Observe that 
\begin{eqnarray*}
k &=& k^{(1)} + t^{(1)} + \sum_{\eta \in {(H_A)}_{*} \setminus 1,\,\eta\neq1} \left(k^{(\eta)} + t^{(\eta)} \right),\\
\Big|X^{(w)}\Big| &=& u^{(1)} + k^{(1)} + v^{(1)} + t^{(1)} + \sum_{\eta \in {(H_A)}_{*} \setminus 1,\,\eta\neq1} \left(u^{(\eta)} + v^{(\eta)} \right).
\end{eqnarray*}

Hence, the contribution to $C_{W_A,S^{(W_A)}}(q)$ from $\left(\lambda^{(1)}, \nu^{(1)}; \left(\mu^{(\eta)}, \gamma^{(\eta)}\right)_{\eta \in {(H_A)}_{*} \setminus 1,\,\eta\neq1} \right)$ is $$\left(q^{u^{(1)}}q^{v^{(1)}} \prod_{\eta \in {(H_A)}_{*} \setminus 1,\,\eta\neq1} q^{u^{(\eta)}}q^{v^{(\eta)}} \right).$$

It follows that 
\begin{eqnarray*}
C_{W_A,S^{\left(W_A\right)}}(q) &=& \left[ \left(\prod_{u_{1}=1}^\infty \frac{1}{1-q^{u_{1}}} \right) \left(\frac{1}{2} \prod_{v_{1}=1}^\infty \frac{1}{1-q^{v_{1}}} + \frac{1}{2} \prod_{v_{1}=1}^\infty \frac{1}{1+q^{v_{1}}} \right) \right]\\
&&\times \prod_{\eta \in {(H_A)}_{*} \setminus 1,\,\eta\neq1} \left[ \left(\prod_{u_{\eta}=1}^\infty \frac{1}{1-q^{u_{\eta}}} \right) \left(\frac{1}{2} \prod_{v_{\eta}=1}^\infty \frac{1}{1-q^{v_{\eta}}} + \frac{1}{2} \prod_{v_{\eta}=1}^\infty \frac{1}{1+q^{v_{\eta}}} \right) \right]\\\\\\\\
&=& \left[ \left( \frac{1}{2} \prod_{n_{1}=1}^\infty \frac{1}{1-q^{2 n_{1}}} + \frac{1}{2} \prod_{n_{1}=1}^\infty \frac{1}{(1-q^{n_{1}})^{2}} \right) \right]\\
&&\times  \prod_{\eta \in {(H_A)}_{*} \setminus 1,\,\eta\neq1} \left[ \left( \frac{1}{2} \prod_{n_{\eta}=1}^\infty \frac{1}{1-q^{2 n_{\eta}}} + \frac{1}{2} \prod_{n_{\eta}=1}^\infty \frac{1}{(1-q^{n_{\eta}})^{2}} \right) \right]\\
&=& \left( \frac{1}{2} \prod_{k=1}^\infty \frac{1}{1-q^{2 k}} + \frac{1}{2} \prod_{k=1}^\infty \frac{1}{(1-q^{k})^{2}} \right)^{\left|{(H_A)}_{*}\right|}.
\end{eqnarray*}
The equality between the first and second line is given in the appendix of \cite{BdlH}.
\end{proof}

\vspace{.3cm}

The \emph{generalized partition function} $p(n)_\textbf{e}$ is defined for the vector $\textbf{e}=(e_1,\dots,e_k)\in\Z^k$ by its generating function
\begin{equation}\tag{2.1}
\sum\limits_{n=0}^\infty p(n)_\textbf{e}q^n=\prod_{n=1}^\infty\frac{1}{(1-q^n)^{e_1}\cdots(1-q^{kn})^{e_k}}.\label{gpf}
\end{equation}
The following theorem gives an asymptotic formula for the generalized partition function, which was obtained by using properties of modular forms\footnote{For background on modular forms, see \cite{WoM}.}.
\begin{thm}[Cotron-Dicks-Fleming \cite{REU}]\label{CDF}
Let $\textbf{e}=(e_1,\dots,e_k)$ be any nonzero vector with nonnegative integer entries, and let $d:=\gcd\{m:e_m\neq0\}$.  Define the quantities
\begin{equation*}
\g:=\g(\textbf{e})=\sum_{m=1}^ke_{dm}\hspace{.5cm}\text{and}\hspace{.5cm}\d:=\d(\textbf{e})=\sum_{m=1}^k\frac{e_{dm}}{m}.
\end{equation*}
Then as $n\rightarrow\infty$, we have that
\begin{equation}\tag{2.2}
p(dn)_{\textbf{e}}\sim\frac{\l A^{\frac{1+\g}{4}}}{2\sqrt{\pi}n^{\frac{3+\g}{4}}}e^{2\sqrt{An}},\label{asymp}
\end{equation}
where
\begin{equation*}
\l:=\prod_{m=1}^k\left(\frac{m}{2\pi}\right)^{\frac{e_{dm}}{2}}\hspace{.5cm}\text{and}\hspace{.5cm}A:=\frac{\pi^2\d}{6}.
\end{equation*}
\end{thm}
Corollaries \ref{asymptotic}, \ref{altasymptotic}, \ref{ratio}, and \ref{L} all follow from the above theorem.

\vspace{.3cm}

\begin{proof}[A Remark on Corollary \ref{altasymptotic}]
By the binomial theorem applied to the conjugacy growth series in equation (\ref{cgswralt}), we find that 
\begin{equation*}
\g_{W_A}(n)\sim\frac{1}{2^{M_A}}\sum_{k=0}^{M_A}\left[\frac{(4M_A-3k)^{\frac{1+2M_A-k}{4}}}{2^{\frac{4M_A-3k+3}{2}}3^{\frac{1+2M_A-k}{4}}n^{\frac{3+2M_A-k}{4}}}\cdot e^{2\pi\sqrt{\left(\frac{4M_A-3k}{12}\right)n}}\right].
\end{equation*}
But, intuitively, the summands corresponding to $k>0$ grow much more slowly than the summand corresponding to $k=0$, since the instance of $k$ in the exponential function is negative.  Therefore, the above sum is asymptotic to the $k=0$ term, so we have
\begin{equation*}
\g_{W_A}(n)\sim\frac{M_A^{\frac{1+2M_A}{4}}}{2^{1+2M_A}3^{\frac{1+2M_A}{4}}n^{\frac{3+2M_A}{4}}}\cdot e^{2\pi\sqrt{\frac{nM_A}{3}}}.
\end{equation*}
\end{proof}

\vspace{.3cm}

We now introduce the proof of Theorems \ref{thm} and \ref{thm2}.  In a paper by Bruinier, Kohnen, and Ono \cite{BKO}, the universal polynomial $F_n$ is defined as\\
\begin{align*}
F_n\left(x_1,\dots,x_{n-1}\right)&\hspace{.2cm}:=\hspace{.2cm}-\frac{2x_1\sigma_1(n-1)}{n-1}\hspace{.3cm}\\
+&\sum_{\substack{m_1,\dots,m_{n-2}\geq0 \\
m_1+\cdots+(n-2)m_{n-2}=n-1} }(-1)^{m_1+\cdots+m_{n-2}}\cdot\frac{(m_1+\cdots+m_{n-2}-1)!}{m_1!\cdots m_{n-2}!}\cdot x_2^{m_1}\cdots x_{n-1}^{m_{n-2}},
\end{align*}
and it is used to define a recursion relation for coefficients of meromorphic modular forms on $SL_2(\Z)$.  Frechette and the author \cite{FL} modify this polynomial to the above $\widehat{F}_n$ and use it to define a recursion relation for coefficients of quotients of Rogers-Ramanujan-type $q$-series.  Their proof surprisingly only requires properties of logarithmic derivatives applied to a $q$-series infinite product identity.  The proof below is adapted from the proof in \cite{FL} and can be applied to any $q$-series infinite product identity, including the famous identity of Nekrasov and Okounkov \cite{NO}.

\begin{proof}[Proof of Theorem \ref{thm}]
Define the $q$-series identity
\begin{equation*}
F_r(q):=\sum\limits_{n=0}^\infty p_n(r)q^n:=\prod_{n=1}^\infty(1-q^n)^r
\end{equation*}
so that $p_n(r)=\g_{W_S}(n)$ and $r=-M_S$.  We take logarithms of both sides to obtain
\begin{eqnarray*}
\log\left(1+\sum\limits_{n=1}^\infty p_n(r)q^n\right)&=&\sum\limits_{n=1}^\infty r\log(1-q^n)\\
&=&-\sum\limits_{n=1}^\infty\sum\limits_{k=1}^\infty\frac{rq^{kn}}{k},
\end{eqnarray*}
by the Taylor expansion for $\log(1-x)$.  Then we take the derivatives of both sides to obtain
\begin{eqnarray*}
\frac{\sum\limits_{n=1}^\infty np_n(r)q^{n-1}}{1+\sum\limits_{n=1}^\infty p_n(r)q^n}&=&-\sum\limits_{n=1}^\infty\sum\limits_{d\mid n} rdq^{n-1}\\
&=&-\sum\limits_{n=1}^\infty r\sigma_1(n)q^{n-1},
\end{eqnarray*}
so we have
\begin{equation*}
\sum\limits_{n=1}^\infty np_n(r)q^n=\left(-\sum\limits_{n=1}^\infty r\sigma_1(n)q^n\right)\left(1+\sum\limits_{n=1}^\infty p_n(r)q^n\right).
\end{equation*}
For convenience, define $b(n):=r\sigma_1(n)$.  Expanding the right hand side and equating coefficients, we now have
\begin{equation*}
0=b(n)+b(n-1)p_1(r)+b(n-2)p_2(r)+\cdots+b(1)p_{n-1}(r)+np_n(r).
\end{equation*}
The symmetric power functions
\begin{equation*}
s_i:=X_1^i+\cdots+X_n^i
\end{equation*}
and the elementary symmetric functions
\begin{equation*}
\sigma_i=\sum\limits_{1\leq j_1\leq\cdots\leq j_i\leq n}X_{j_1}\cdots X_{j_i}
\end{equation*}
exhibit a similar relationship; namely, we have the identity
\begin{equation}\label{eq}
0=s_n-s_{n-1}\sigma_1+s_{n-2}\sigma_2-\cdots+(-1)^{n-1}s_1\sigma_{n-1}+(-1)^n\sigma_n.\tag{2.1}
\end{equation}
Evaluating equation (\ref{eq}) at $(X_1,\dots,X_n)=(\l(1,n),\dots,\l(n,n))$, where $\l(j,n)$ are the roots of the polynomial
\begin{equation*}
X^n+p_1(r)X^{n-1}+\cdots+p_{n-1}(r)X+p_n(r),
\end{equation*}
we have that $p_n(r)=\sigma_n$ for $n\geq1$.  Then we have $b(n)=(-1)^ns_n$.  Using the fact that
\begin{equation*}
s_n=n\sum_{\substack{m_1,\dots,m_n\geq0 \\
m_1+\cdots+nm_n=i} }(-1)^{m_2+m_4+\cdots}\cdot\frac{(m_1+\cdots+m_n-1)!}{m_1!\cdots m_n!}\cdot\sigma_1^{m_1}\cdots\sigma_n^{m_n},
\end{equation*}
we arrive at the recursion
\begin{equation*}
p_n(r)=\widehat{F}_n\big(p_1(r),\dots,p_{n-1}(r)\big)-\frac{r}{n}\sigma_1(n).
\end{equation*}
Thus, we have
\begin{equation*}
\g_{W_S}(n)=\widehat{F}_n\big(\g_{W_S}(1),\dots,\g_{W_S}(n-1)\big)+\frac{M_S}{n}\sigma_1(n).
\end{equation*}
\end{proof}
\vspace{.5cm}
Theorem \ref{thm} gives a recurrence formula for the coefficients $\g_{W_S}(n)$ of the conjugacy growth series of a permutational wreath product in which the group $H_S$ has $M_S$ conjugacy classes.  Now, we consider the more general infinite product $\prod_{n\geq 1}\left(1-q^n\right)^r$ for any complex number $r$, and we ignore its implications for finite groups.  Then the above proof also applies to the coefficients of the Nekrasov-Okounkov hook length formula \cite{NO}
\begin{equation*}
\sum\limits_{\l\in\mathcal{P}}x^{|\l|}\prod\limits_{h\in\mathcal{H}(\l)}\left(1-\frac{z}{h^2}\right)=\prod\limits_{k\geq1}\left(1-x^k\right)^{z-1}
\end{equation*}
if we change variables $z\mapsto1+r$ and $x\mapsto q:=e^{2\pi i\t}$ for $\t\in\mathcal{H}$.  The coefficients
\begin{equation*}
\prod\limits_{h\in\mathcal{H}(\l)}\left(1-\frac{z}{h^2}\right)=\prod\limits_{h\in\mathcal{H}(\l)}\left(1-\frac{1+r}{h^2}\right)
\end{equation*}
of the infinite product
\begin{equation*}
\prod\limits_{k\geq1}\left(1-x^k\right)^{z-1}=\prod\limits_{n\geq1}\left(1-q^n\right)^r
\end{equation*}
therefore satisfy the recurrence relation
\begin{equation*}
\prod\limits_{h\in\mathcal{H}(\l)}\left(1-\frac{1+r}{h^2}\right)=\g_{W_S}(n)=\widehat{F}_n\big(\g_{W_S}(1),\dots,\g_{W_S}(n-1)\big)-\frac{r}{n}\sigma_1(n).
\end{equation*}
Although for $r\in\C\setminus\Z^+$ we can no longer observe the relationship between the number of conjugacy classes of $H_S$ and the coefficients of the conjugacy growth series of $H_S\wr_X\text{Sym}(X)$, we do obtain a simple recursion for the Nekrasov-Okounkov hook length formula which is independent of complex analysis and hook lengths.

\begin{proof}[Proof of Theorem \ref{thm2}]
This proof closely follows the proof of Theorem \ref{thm}.  Define the $q$-series identity
\begin{equation*}
F_{M_A}(q):=\sum\limits_{n=0}^\infty P_n(M_A)q^n:=\left(\frac{1}{2}\prod_{n=1}^\infty\frac{1}{\left(1-q^n\right)^2}+\frac{1}{2}\prod_{n=1}^\infty\frac{1}{1-q^{2n}}\right)^{M_A}
\end{equation*}
so that $P_n(M_A)=\g_{W_A}(n)$.  Then, by the binomial theorem, we have
\begin{equation*}
\sum\limits_{n=0}^\infty P_n(M_A)q^n=\frac{1}{2^{M_A}}\sum_{k=1}^{M_A}\binom{M_A}{k}\prod_{n=1}^\infty\frac{1}{\left(1-q^n\right)^{2k}\left(1-q^{2n}\right)^{M_A-k}}.
\end{equation*}
It suffices to find recurrence relations for each summand.  Define
\begin{equation*}
F_{M_A,k}(q):=\sum_{n=0}^\infty a_k(n)q^n:=\prod_{n=1}^\infty\frac{1}{\left(1-q^n\right)^{2k}\left(1-q^{2n}\right)^{M_A-k}}.
\end{equation*}
We take the logarithmic derivative of both sides as in the proof of Theorem \ref{thm}.  First, we take logarithms of both sides to obtain
\begin{eqnarray*}
\log\left(1+\sum\limits_{n=1}^\infty a_k(n)q^n\right)&=&-2k\sum_{n=1}^\infty\log\left(1-q^n\right)+(k-M_A)\sum_{n=1}^\infty\log\left(1-q^{2n}\right)\\
&=&-(k+M_A)\sum_{n=1}^\infty\log\left(1-q^n\right)+(k-M_A)\sum_{n=1}^\infty\log\left(1+q^n\right)\\
&=&(k+M_A)\sum_{n=1}^\infty\sum_{m=1}^\infty\frac{q^{mn}}{m}+(M_A-k)\sum_{n=1}^\infty\sum_{m=1}^\infty\frac{(-1)^mq^{mn}}{m},
\end{eqnarray*}
by the Taylor expansions for $\log(1-x)$ and $\log(1+x)$.  Then we take the derivatives of both sides to obtain
\begin{eqnarray*}
\frac{\sum\limits_{n=1}^\infty na_k(n)q^{n-1}}{1+\sum\limits_{n=1}^\infty a_k(n)q^n}&=&-\sum_{n=1}^\infty\sum_{d\mid n}d\cdot\left[(-1)^{\frac{n}{d}}(k-M_A)-(k+M_A)\right]q^{n-1},
\end{eqnarray*}
so we have
\begin{equation*}
\sum\limits_{n=1}^\infty na_k(n)q^n=\left(-\sum_{n=1}^\infty\sum_{d\mid n}d\cdot\left[(-1)^{\frac{n}{d}}(k-M_A)-(k+M_A)\right]q^n\right)\left(1+\sum\limits_{n=1}^\infty a_k(n)q^n\right).
\end{equation*}
For convenience, define $b_k(n):=\sum_{d\mid n}d\cdot\left[(-1)^{\frac{n}{d}}(k-M_A)-(k+M_A)\right]$.  Expanding the right hand side and equating coefficients, we now have
\begin{equation*}
0=b_k(n)+b_k(n-1)a_k(1)+b_k(n-2)a_k(2)+\cdots+b_k(1)a_k(n-1)+na_k(n).
\end{equation*}
Using the same identity between the symmetric power functions and the elementary symmetric functions as in the proof of Theorem \ref{thm}, we arrive at the recursion
\begin{eqnarray*}
a_k(n)&=&\widehat{F}_n\big(a_k(1),\dots,a_k(n-1)\big)-\frac{1}{n}\sum_{d\mid n}d\cdot\left[(-1)^{\frac{n}{d}}(k-M_A)-(k+M_A)\right]\\
&=&\widehat{F}_n\big(a_k(1),\dots,a_k(n-1)\big)-\sum_{\d\mid n}\d\cdot\left[(-1)^\d(k-M_A)-(k+M_A)\right].
\end{eqnarray*}
Thus, we have
\begin{equation*}
\g_{W_A}(n)=\frac{1}{2^{M_A}}\sum_{k=0}^{M_A}\binom{M_A}{k}\left(\widehat{F}_n\big(a_k(1),\dots,a_k(n-1)\big)-\sum_{\d\mid n}\d\cdot\left[(-1)^\d(k-M_A)-(k+M_A)\right]\right).
\end{equation*}
\end{proof}

\begin{rmk}
This recurrence relation gives the coefficients $\g_{W_A}(n)$ in terms of the coefficients $a_k(1)$, $\dots$, $a_k(n-1)$ of each summand.  Since the linear combination of infinite products is raised to the $(M_A)$th power in the conjugacy growth series, presumably there is no simple way to obtain a recurrence relation for $\g_{W_A}(n)$ in terms of $\g_{W_A}(1),\dots,\g_{W_A}(n-1)$ as in the symmetric case.
\end{rmk}

\end{document}